\documentclass[12pt,twoside,a4paper]{article}
\usepackage{amsmath}
\usepackage{amssymb}
\usepackage{float}
\usepackage{graphicx}
\usepackage{epsfig}
\usepackage{cite}
\usepackage{epstopdf}
\usepackage{caption}
\usepackage{subfigure}
\usepackage{booktabs}
\usepackage{latexsym}
\usepackage{stmaryrd}
\usepackage{multicol}
\usepackage{array}
\usepackage{color}
\usepackage{amsthm}
\usepackage{amsfonts}
\usepackage{changes}
\usepackage{xcolor}
\usepackage{soul}

\graphicspath{{picture/}}
\usepackage[top=1in,bottom=1in,left=1.25in,right=1.25in]{geometry}
\usepackage[colorlinks,linkcolor=red,anchorcolor=green,citecolor=blue,CJKbookmarks=True]{hyperref}
\numberwithin{equation}{section} \makeatletter
\newtheoremstyle{theorem-mystyle}{2pt}{2pt}{\rm}{3mm}{\bf}{~~}{3mm}{}
\numberwithin{equation}{section}
\newtheorem{dl}{Theorem}[section]
\newtheorem{yl}{Lemma}[section]
\newtheorem{dy}{Definition}[section]
\newtheorem{tl}{Corollary}[section]

\date{}
\begin{document}
\author{{\normalsize
Licai Shao, Xingong Zhang\thanks{ Corresponding author,
	E-mail: zhangxingong@cqnu.edu.cn}, Yisheng Song} \\
{\small School of Mathematical Sciences, Chongqing Normal University,}\\
{\small  Chongqing, 401331, People's Republic of China}\\}
\title{\textbf{New Pareto Z-eigenvalue localization theorems for tensors with applications}
\thanks{ The project is supported by the National Natural
Science Foundation of China (No. 12171064).}}
\maketitle \noindent \small {\normalsize{\bf Abstract:}}
 In this paper, we establish some new localization sets for Pareto Z-eigenvalues of tensor eigenvalue complementarity problems and discuss the relationships among existing ones. As applications, a sufficient condition for the solution existence of the tensor complementarity problem (TCP) and the tensor generalized eigenvalue complementarity problem (TGEiCP) are presented. In addition, this condition is also used to test the vacuum stability of a general scalar potential.
\section{Introduction}
For any positive integer $n$, denote $[n]:=\{1,2,\ldots,n\}.$ Given two integers $m,n\geq2,$ an $m$th order $n$-dimensional real tensor is a multi-array of entries $\mathcal C=(c_{i_{1}i_{2}\cdots i_{m}})\in \mathbb{R},$ where $i_{j}\in[n]$ for all $j\in[m].$ The set of all $m$th order $n$-dimensional real tensors is denoted by $\mathbb{\mathbb{R}}^{[m,n]}.$ An $m$th order $n$-dimensional tensor $\mathcal C$ is said to be symmetric if its entries $c_{i_{1},i_{2},\cdots,i_{m}}$ are invariant for any permutation of the indices. For any $\mathcal C=(c_{i_{1}i_{2}\cdots i_{m}})\in\mathbb{\mathbb{R}}^{[m,n]}$ and $x=(x_{1},x_{2},\ldots,x_{n})^{\top}\in\mathbb{R}^{n},$ $\mathcal Cx^{m}$ is defined by
\begin{equation*}
\mathcal C x^{m} = \sum_{i_{1},i_{2},\cdots,i_{m}\in[n]} c_{i_{1} i_{2} \dots i_{m}} x_{i_{1}}x_{i_{2}} \cdots x_{i_{m}};
\end{equation*}
$\mathcal Cx^{m-1}$ denotes a vector whose $i$th component is given by
\begin{equation*}
(\mathcal{C}x^{m-1})_{i} = \sum_{i_{2},\cdots,i_{m}\in[n]} c_{i i_2 \dots i_m} x_{i_2} \cdots x_{i_m}.
\end{equation*}
For an $m$th order $n$-dimensional symmetric tensor $\mathcal C$, we consider a constrained optimization problem of the form:
\begin{equation}\label{1.1}
\begin{aligned}
\min \; & \frac{1}{m} \mathcal C x^{m} \\
\text{s.t.} \; & x^{\top} x = 1, \\
& x \in \mathbb{R}_{+}^{n},
\end{aligned}
\end{equation}
where, $\mathbb{R}_{+}^{n}=\{x\in\mathbb{R}^{n}: x\geq0\}.$\\
The Lagrange function for problem (\ref{1.1}) is then given by
\begin{equation*}
L(x, \lambda, y) = \frac{1}{m} \mathcal C x^{m} + \frac{1}{m} \lambda \left( 1 - x^{\top} x \right) - x^{\top} y,
\end{equation*}
where \( x, y \in \mathbb{R}_{+}^{n} \), \( \frac{\lambda}{m} \in \mathbb{R} \) is the Lagrange multiplier of the equality constraint and \( y \) is the Lagrange multiplier of non-negative constraint. Then the solution to problem (\ref{1.1}) satisfies the following conditions:
\begin{equation*}
\begin{aligned}
& y=\lambda x-\mathcal C x^{m-1}\geq0,\\
&\mathcal C x^{m}=\lambda x^{\top} x,\\
& x \in \mathbb{R}_{+}^{n}.
\end{aligned}
\end{equation*}
The Pareto Z-eigenpair of the tensor eigenvalue complementarity problem was introduced by Song et al.\cite{Song1}, which is a natural generalization of the matrix eigenvalue complementarity problem \cite{He1, S1, A1, F1, Qi1, L1, Qi2}.
\begin{dy}\label{dy:1.1}
Let $\mathcal C=(c_{i_{1}\cdots i_{m}})\in\mathbb{ R}^{[m,n]},$ if $x=(x_{1},x_{2}\cdots,x_{m})^{\top}\in\mathbb{ R}^{n}\setminus\{0\}$ is a nonzero real vector, $\lambda\in\mathbb{ R}$ is a real number, which satisfies the following tensor Z-eigenvalue complementarity problem:
\begin{equation*}
0\leq x\perp(\lambda x-\mathcal Cx^{m-1})\geq0\ \mbox{and} \ x^{\top}x=1,
\end{equation*}
then $\lambda$ is said to be a Pareto Z-eigenvalue of $\mathcal C$, $x$ is called the corresponding Pareto Z-eigenvector.
\end{dy}
Pareto Z-eigenvalues are also called complementarity Z-eigenvalues. It is worth noting that Zeng \cite{Zeng1} proposed  a tensor may not have the Z-eigenvalues, but has the Pareto Z-eigenvalues. Meanwhile, Zeng \cite{Zeng1} presented a semidefinite relaxation algorithm to compute Pareto Z-eigenvalues of tensor eigenvalue complementarity problems. However, computing all Pareto Z-eigenvalues of tensors is always a cumbersome task. Therefore, eigenvalue localization sets are introduced as an efficient tool to characterize the distribution of eigenvalues for higher-order tensors\cite{Wang1, Xu1, Xu2, Xu3, Xu4}. Based the above facts, some Pareto Z-eigenvalue localization sets for a piezoelectric tensor are given in\cite{Yang1, He2, Wang2}. For instance, Yang et al.\cite{Yang1} established an important localization theorem characterizing the distribution of Pareto Z-eigenvalues of tensors. 
\begin{dl}\label{dl:1.1}
Let $\mathcal C\in\mathbb{R}^{[m,n]}$ and $\sigma(\mathcal C)\neq\emptyset.$ Then,
\begin{equation*}
\sigma(\mathcal C)\subseteq\bigcup_{i\in[n]}\Omega_{i}(\mathcal C):=\{\lambda\in\mathbb{R}:|\lambda|\leq \max\{R_{i}(\mathcal C)_{+},R_{i}(\mathcal C)_{-}\}\}.
\end{equation*}
\end{dl}
Inspired by the results obtain in\cite{Yang1, He2, Xu1, Xu2, Xu3, Xu4, Sang1, Wang2, Du1}, we establish shaper localization sets than existing ones for Pareto Z-eigenvalues, compare the relationships among these sets, and present sufficient conditions for determining  copositivity (strict copositivity) of real symmetric tensors.

The remainder of this paper is organized as follows. In Section 2, new Pareto Z-eigenvalue localization theorems are established. Furthermore, the relationships between the newly obtained localization sets and some existing results are analyzed to show that our results are sharper. Section 3 presents verifiable sufficient conditions for tensor copositivity (strict copositivity) derived from the established localization sets. Furthermore, we establish conditions ensuring the solvability of both the tensor complementarity problem (TCP) and the tensor generalized eigenvalue complementarity problem (TGEiCP). These results additionally enable the verification of vacuum stability for general scalar potentials. In Section 4, the conclusions of this paper are summarized.

To conclude this section, some necessary notations as follows. The Kronecker symbol of m indices is defined by
\begin{equation*}
\delta_{i_{1} i_{2}\cdots i_{m}} =
\begin{cases} 
1, & \text{if } i_{1} = i_{2} = \cdots = i_{m},\\
0, & \text{otherwise.}
\end{cases}
\end{equation*}
For any $a\in\mathbb{R}$ and $\mathcal C=(c_{i_{1}\cdots i_{m}})\in\mathbb{ R}^{[m,n]}$, define
\begin{equation*}
[a]_{+}=:\max\{0,a\}, \quad [a]_{-}=:\max\{0,-a\}, \quad [\mathcal C]_{+}:=[c_{i_{1}\cdots i_{m}}]_{+}\in\mathbb{ R}^{[m,n]}, \quad [\mathcal C]_{-}:=[c_{i_{1}\cdots i_{m}}]_{-}\in\mathbb{ R}^{[m,n]}.
\end{equation*}
For any $i,j\in[n],$ let
\begin{equation*}
R_i(\mathcal{C})_{+}:=\sum_{i_2,\dots,i_m \in [n]}[c_{ii_2\dots i_m}]_{+}, \quad
R_i(\mathcal{C})_{-}:=\sum_{i_2,\dots,i_m \in [n]}[c_{ii_2\dots i_m}]_{-},
\end{equation*}
$$P_{i}^{j}(\mathcal C)_{+}:=\sum_{\substack{i_{2},\cdots,i_{m}\in [n] \\ j\notin\{i_{2},\cdots,i_{m}\}}}[c_{ii_{2}\cdots i_{m}}]_{+},\ P_{i}^{j}(\mathcal C)_{-}:=\sum_{\substack{i_{2},\cdots,i_{m}\in [n] \\ j\notin\{i_{2},\cdots,i_{m}\}}}[c_{ii_{2}\cdots i_{m}}]_{-}.$$
\section{New Pareto Z-eigenvalue localization sets}
To prove the main results, we reviewed some existing facts, which will be useful in the subsequent analysis.
\begin{yl}\label{yl:2.1}\cite{Song1}
A tensor $\mathcal C$ is  copositive (strictly copositive) if and only if all the Pareto Z-eigenvalues of $\mathcal C$ are nonegative (positive).
\end{yl}
\begin{yl}\label{yl:2.2}\cite{Xu4}
Let $u$ and $v$ be two nonnegative real numbers, then $|u-v|\leq max\{u,v\}$.
\end{yl}
\begin{yl}\label{yl:2.3}\cite{Sang1}
Let $x_{1}^{2}+x_{2}^{2}+\cdots+x_{n}^{2}=1,$ where $x_{i}\in\mathbb{R}, i\in[n].$ If $y_{1},\ldots,y_{k}$ are arbitrary $k$ entries of $x_{1},\ldots,x_{n},$ then
\begin{equation*}
 |y_{1}||y_{2}|\cdots|y_{k}|\leq\frac{1}{k^{\frac{k}{2}}}.
\end{equation*}
\end{yl}
\begin{dl}\label{dl:2.1}
Let $\mathcal C=(c_{i_{1}\cdots i_{m}})\in\mathbb{ R}^{[m,n]}$. Then
$$\sigma(\mathcal C)\subseteq \mathcal O(\mathcal C)=\bigcup_{i\in[n]}\bigcap_{j\in[n] \ i\neq j}\mathcal O_{i,j}(\mathcal C),$$
where $\mathcal O_{i,j}(\mathcal C)=\{\lambda\in\mathbb{R}: \left(|\lambda|-\max\{R_{i}(\mathcal C)_{+}-P^{j}_{i}(\mathcal C)_{+}, R_{i}(\mathcal C)_{-}-P^{j}_{i}(\mathcal C)_{-}\}\right)|\lambda|\leq\max\{P^{j}_{i}(\mathcal C)_{+},P^{j}_{i}(\mathcal C)_{-}\}\times\max\{R_{j}(\mathcal C)_{+},R_{j}(\mathcal C)_{-}\}$.
\end{dl}
\begin{proof}
Suppose that $\lambda$ is a Pareto Z-eigenvalue of $\mathcal C$ with corresponding eigenvector $x$. Let $|x_{p}|\geq|x_{q}|\geq \max \{x_{i}:i\in[n], i\neq p, i\neq q\}$. By Definition \ref{dy:1.1}, we have
\begin{equation}\label{2.1}
\lambda x_{p}^{2}=\sum_{\substack{i_{2},\cdots,i_{m}\in [n] \\ p\in\{i_{2},\cdots,i_{m}\}}}c_{pi_{2}\cdots i_{m}}x_{p}x_{i_{2}}\cdots x_{i_{m}}+\sum_{\substack{i_{2},\cdots,i_{m}\in [n] \\ p\notin\{i_{2},\cdots,i_{m}\}}}c_{pi_{2}\cdots i_{m}}x_{p}x_{i_{2}}\cdots x_{i_{m}}.
\end{equation}
Taking the absolute value of both sides of Equation (\ref{2.1}), it holds that
\begin{equation*}
\begin{aligned}
|\lambda|x_{p}^{2}&=\Bigg|\sum_{\substack{i_{2},\cdots,i_{m}\in [n] \\ q\in\{i_{2},\cdots,i_{m}\}}}c_{pi_{2}\cdots i_{m}}x_{p}x_{i_{2}}\cdots x_{i_{m}}+\sum_{\substack{i_{2},\cdots,i_{m}\in [n] \\ q\notin\{i_{2},\cdots,i_{m}\}}}c_{pi_{2}\cdots i_{m}}x_{p}x_{i_{2}}\cdots x_{i_{m}}\Bigg|\\
&\leq\Bigg|\sum_{\substack{i_{2},\cdots,i_{m}\in [n] \\ q\in \{i_{2},\cdots,i_{m}\}}}[c_{pi_{2}\cdots i_{m}}]_{+}x_{p}x_{i_{2}}\cdots x_{i_{m}}-\sum_{\substack{i_{2},\cdots,i_{m}\in [n] \\ q\in \{i_{2},\cdots,i_{m}\}}}[c_{pi_{2}\cdots i_{m}}]_{-}x_{p}x_{i_{2}}\cdots x_{i_{m}}\Bigg|\\
&+\Bigg|\sum_{\substack{i_{2},\cdots,i_{m}\in [n] \\ q\notin \{i_{2},\cdots,i_{m}\}}}[c_{pi_{2}\cdots i_{m}}]_{+}x_{p}x_{i_{2}}\cdots x_{i_{m}}-\sum_{\substack{i_{2},\cdots,i_{m}\in [n] \\ q\notin \{i_{2},\cdots,i_{m}\}}}[c_{pi_{2}\cdots i_{m}}]_{-}x_{p}x_{i_{2}}\cdots x_{i_{m}}\Bigg|\\
&\leq\max\left\{\sum_{\substack{i_{2},\cdots,i_{m}\in [n] \\ q\in \{i_{2},\cdots,i_{m}\}}}[c_{pi_{2}\cdots i_{m}}]_{+},\sum_{\substack{i_{2}\cdots i_{m}\in [n] \\ q\in \{i_{2},\cdots,i_{m}\}}}[c_{pi_{2}\cdots i_{m}}]_{-}\right\}x_{p}^{2}\\
&+\max\left\{\sum_{\substack{i_{2},\cdots,i_{m}\in [n] \\ q\notin \{i_{2},\cdots,i_{m}\}}}[c_{pi_{2}\cdots i_{m}}]_{+},\sum_{\substack{i_{2},\cdots,i_{m}\in [n] \\ q\notin \{i_{2},\cdots,i_{m}\}}}[c_{pi_{2}\cdots i_{m}}]_{-}\right\}x_{p}x_{q},
\end{aligned}
\end{equation*}
that is
\begin{equation}\label{2.2}
\left(|\lambda|-\max\{R_{p}(\mathcal C)_{+}-P_{p}^{q}(\mathcal C)_{+},R_{p}(\mathcal C)_{-}-P_{p}^{q}(\mathcal C)_{-}\}\right)x_{p}^{2}\leq\max\{P_{p}^{q}(\mathcal C)_{+},P_{p}^{q}(\mathcal C)_{-}\}x_{p}x_{q}.
\end{equation}
From Definition \ref{dy:1.1}, we have
\begin{eqnarray*}
|\lambda| x_{q}^{2} &=& \Bigg|\sum_{\substack{i_{2},\cdots,i_{m}\in [n]}}[c_{qi_{2}\cdots i_{m}}]_{+}x_{q}x_{i_{2}}\cdots x_{i_{m}}-\sum_{\substack{i_{2},\cdots,i_{m}\in [n]}}[c_{qi_{2}\cdots i_{m}}]_{-}x_{q}x_{i_{2}}\cdots x_{i_{m}}\Bigg| \\
  &\leq& \max\left\{ 
    \sum_{\substack{i_2,\dots,i_{m} \in [n]}} 
    [c_{q i_{2} \cdots i_{m}}]_{+},\ 
    \sum_{\substack{i_2,\dots,i_{m} \in [n]}} 
    [c_{q i_{2} \cdots i_{m}}]_{-}
\right\} x_{p} x_{q},
\end{eqnarray*}
that is 
\begin{equation}\label{2.3}
|\lambda|x_{q}^{2}\leq\max\{R_{q}(\mathcal C)_{+},R_{q}(\mathcal C)_{-}\}x_{p}x_{q}.
\end{equation}

{\bf Case I.} If $x_{q}=0.$ It follows from \eqref{2.3} that
\begin{equation*}
(|\lambda|-\max\{R_{p}(\mathcal C)_{+}-P_{p}^{q}(\mathcal C)_{+},R_{p}(\mathcal C)_{-}-P_{p}^{q}(\mathcal C)_{-}\})|\lambda|\leq0\leq\max\{P_{p}^{q}(\mathcal C)_{+},P_{p}^{q}(\mathcal C)_{-}\}\times\max\{R_{q}(\mathcal C)_{+},R_{q}(\mathcal C)_{-}\},
\end{equation*}
this shows $\lambda\in\mathcal O_{p,q}(\mathcal C)\subseteq\mathcal O(\mathcal C).$

{\bf Case II.} If $x_{q}>0.$ 
Multiplying \eqref{2.2} and \eqref{2.3}, we have
$$\left(|\lambda|-\max\{R_{p}(\mathcal C)_{+}-P_{p}^{p}(\mathcal C)_{+},R_{p}(\mathcal C)_{-}-P_{p}^{p}(\mathcal C)_{-}\}\right)|\lambda|\leq\max\{P_{p}^{p}(\mathcal C)_{+},P_{p}^{p}(\mathcal C)_{-}\}\times\max\{R_{q}(\mathcal C)_{+},R_{q}(\mathcal C)_{-}\},$$
this implies $\lambda\in\mathcal O_{p,q}(\mathcal C)\subseteq\mathcal O(\mathcal C).$
\end{proof}
\begin{dl}\label{dl:2.2}
Let $\mathcal C=(c_{i_{1}\cdots i_{m}})\in\mathbb{ R}^{[m,n]},$ $m, n\geq2$. Then
$$\sigma(\mathcal C)\subseteq\mathcal O(\mathcal C)\subseteq\Omega(\mathcal C).$$
\end{dl}
\begin{proof}
For any $\lambda \in \mathcal O_{p,q}(\mathcal C)$, we divide the proof into two parts. There exist $p,q\in[n], p\neq q$ such that $\lambda \in\mathcal O(\mathcal C),$ that is
$$(|\lambda|-\max\{R_{p}(\mathcal C)_{+}-P_{p}^{q}(\mathcal C)_{+},R_{p}(\mathcal C)_{-}-P_{p}^{q}(\mathcal C)_{-}\})|\lambda|$$
$$\leq\max\{P_{p}^{q}(\mathcal C)_{+},P_{p}^{q}(\mathcal C)_{-}\}\times\max\{R_{q}(\mathcal C)_{+},R_{q}(\mathcal C)_{-}\}.$$

{\bf Case I.} If $\max\{P_{p}^{q}(\mathcal C)_{+},P_{p}^{q}(\mathcal C)_{-}\}\times\max\{R_{q}(\mathcal C)_{+},R_{q}(\mathcal C)_{-}\}=0,$ then
$|\lambda|=0$ or $|\lambda|-\max\{R_{p}(\mathcal C)_{+}-P_{p}^{q}(\mathcal C)_{+},R_{p}(\mathcal C)_{-}-P_{p}^{q}(\mathcal C)_{-}\}\leq0.$ Hence, $\lambda\in\Omega_{p}(\mathcal C)\subseteq\Omega(\mathcal C).$

{\bf Case II.} If $\max\{P_{p}^{q}(\mathcal C)_{+},P_{p}^{q}(\mathcal C)_{-}\}\times\max\{R_{q}(\mathcal C)_{+},R_{q}(\mathcal C)_{-}\}>0,$ then
\begin{equation*}
\frac{|\lambda|-\max\left\{R_{p}(\mathcal{C})_{+}-P_{p}^{q}(\mathcal{C})_{+},\; R_{p}(\mathcal{C})_{-}-
P_{p}^{q}(\mathcal{C})_{-}\right\}}{\max\left\{P_{p}^{q}(\mathcal{C})_{+},\; P_{p}^{q}(\mathcal{C})_{-}\right\}}
\cdot\frac{|\lambda|}{\max\left\{R_{q}(\mathcal{C})_{+},\; R_{q}(\mathcal{C})_{-}\right\}}\leq 1,
\end{equation*}
which implies
\begin{equation*}
\frac{|\lambda| - \max\left\{R_{p}(\mathcal{C})_{+} - P_{p}^{q}(\mathcal{C})_{+},\; R_{p}(\mathcal{C})_{-} - P_{p}^{q}(\mathcal{C})_{-}\right\}}{\max\left\{P_{p}^{q}(\mathcal{C})_{+},\; P_{p}^{q}(\mathcal{C})_{-}\right\}} \leq 1
\quad \text{or} \quad
\frac{|\lambda|}{\max\left\{R_{q}(\mathcal{C})_{+},\; R_{q}(\mathcal{C})_{-}\right\}} \leq 1.
\end{equation*}
Therefore, $\lambda\in\Omega_{p}(\mathcal C)\bigcup\Omega_{q}(\mathcal C)\subseteq\Omega(\mathcal C).$
\end{proof}
\begin{dl}\label{dl:2.3}
Let $\mathcal C=(c_{i_{1}\cdots i_{m}})\in\mathbb{ R}^{[m,n]}, m,n\geq2$ and $\max_{i\in[n]}\{R_{i}(\mathcal C)_{+},R_{i}(\mathcal C)_{-}\}\neq0$ for $i\in[n]$.
$$\sigma(\mathcal C)\subseteq\mathcal F(\mathcal C)=\bigcup_{c_{i_{1}\cdots i_{m}}\neq0}\left\{\lambda\in\mathbb{R}:|\lambda|\leq\prod^{m}_{l=1}\max\left\{\left(R_{i_{l}}(\mathcal C)_{+}\right)^{\frac{1}{m}},\left(R_{i_{l}}(\mathcal C)_{-}\right)^{\frac{1}{m}}\right\}\right\}.$$
\end{dl}
\begin{proof}
Let $\lambda$ be a Pareto Z-eigenvalue of $\mathcal C$ with corresponding eigenvector $x$. According to the Definition \ref{dy:1.1}, then for all $i\in[n],$ it holds that
\begin{equation}\label{2.4}
\lambda x_{i}^{2}=\sum_{i_{2},\cdots,i_{m}\in [n]}c_{ii_{2}\cdots i_{m}}x_{i}x_{i_{2}}\cdots x_{i_{m}}.
\end{equation}
Denote $x_{\alpha}=\max\{x_{i_{1}} x_{i_{2}} \cdots x_{i_{m}}\}:c_{i_{1}i_{2}\cdots i_{m}}\neq0, i_{1},i_{2},\cdots,i_{m}\in[n]\},$ Taking the modulus of \eqref{2.4}, for all $i\in[n],$ we can conclude that
\begin{equation}\label{2.5}
\begin{aligned}
|\lambda| x_{i}^{2}
&= \bigg| \sum_{i_{2},\dots,i_{m} = 1}^{n} c_{ii_{2}\cdots i_{m}} x_{i} x_{i_{2}} \cdots x_{i_{m}} \bigg| \\
&= \bigg| \sum_{i_{2},\dots,i_{m} = 1}^{n} [c_{ii_{2}\cdots i_{m}}]_{+} - \sum_{i_{2},\dots,i_m = 1}^{n} [c_{ii_{2}\cdots i_{m}}]_{-} \bigg|x_{i} x_{i_{2}} \cdots x_{i_{m}} \\
&= \bigg| \sum_{c_{i_{2}\dots i_{m}} \neq 0} [c_{ii_{2}\cdots i_{m}}]_{+} - \sum_{c_{i_{2}\dots i_{m}} \neq 0} [c_{ii_{2}\cdots i_{m}}]_{-} \bigg|x_{i} x_{i_{2}} \cdots x_{i_{m}} \\
&\leq \Big( \sum_{c_{i_{2}\dots i_{m}} \neq 0} [c_{ii_{2}\cdots i_{m}}]_{+} - \sum_{c_{i_{2}\dots i_{m}} \neq 0} [c_{ii_{2}\cdots i_{m}}]_{-}\Big) x_{\alpha}\\
&\leq \max\Big\{R_{i}(\mathcal{C})_{+}, R_{i}(\mathcal{C})_{-}\Big\}x_{\alpha}.
\end{aligned}
\end{equation}
Next, we consider two cases as follows.

{\bf Case I.} If $x_{\alpha}=0.$ Since $x\neq0,$ support that $x_{t}\neq0, t\in[n],$ from \eqref{2.5}, we have $|\lambda| x_{t}^{2}=0$, which implies that $\lambda=0.$ Therefore, we obtain $\lambda\in\mathcal F(\mathcal C).$

{\bf Case II.} $x_{\alpha}>0,$ support that $x_{\alpha}=x_{j_{1}}x_{j_{2}}\cdots x_{j_{m}}.$ From \eqref{2.5}, we deduce the inequalities
\begin{equation*}
|\lambda|x_{j_{1}}^{2}\leq\max\{R_{j_{1}}(\mathcal C)_{+},R_{j_{1}}(\mathcal C)_{-}\}x_{\alpha}
\end{equation*}
\begin{center}
$\vdots$
\end{center}
\begin{equation*}
|\lambda|x_{j_{m}}^{2}\leq\max\{R_{j_{m}}(\mathcal C)_{+},R_{j_{m}}(\mathcal C)_{-}\}x_{\alpha}.
\end{equation*}
Then 
\begin{equation*}
\prod^{m}_{l=1}|\lambda|x_{j_{l}}^{2}\leq x_{\alpha}^{m}\prod^{m}_{l=1}\max\{R_{j_{l}}(\mathcal C)_{+},R_{j_{l}}(\mathcal C)_{-}\},
\end{equation*}
which imply that
\begin{equation*}
x_{\alpha}^{2}\prod^{m}_{l=1}|\lambda|\leq x_{\alpha}^{m}\prod^{m}_{l=1}\max\{R_{j_{l}}(\mathcal C)_{+},R_{j_{l}}(\mathcal C)_{-}\}.
\end{equation*}
Therefore
\begin{equation*}
\prod^{m}_{l=1}|\lambda|\leq x_{\alpha}^{m-2}\prod^{m}_{l=1}\max\{R_{j_{l}}(\mathcal C)_{+},R_{j_{l}}(\mathcal C)_{-}\}\leq\prod^{m}_{l=1}\max\{(R_{j_{l}}(\mathcal C)_{+})^{\frac{1}{m}},(R_{j_{l}}(\mathcal C)_{-})^{\frac{1}{m}}\}.
\end{equation*}
The proof is completed.
\end{proof}
\begin{dl}\label{dl:2.4}
Let $\mathcal C=(c_{i_{1}\cdots i_{m}})\in\mathbb{ R}^{[m,n]}, m, n\geq2$ and $\max_{i\in[n]}\{R_{i}(\mathcal C)_{+},R_{i}(\mathcal C)_{-}\}\neq0$, for $i\in[n],$ then
$$\mathcal F(\mathcal C)\subseteq\Omega(\mathcal C).$$
\end{dl}
\begin{proof}
For any $\lambda\in\mathcal F(\mathcal C),$ If $\lambda\notin\Omega(\mathcal C),$ then $|\lambda|>\max_{i\in[n]}\{R_{i}(\mathcal C)_{+},R_{i}(\mathcal C)_{-}\}$ for each $i\in[n].$ In this case, for all $c_{i_{1}\cdots i_{m}}\neq0, i_{j}\in[n]$ and $j\in1,2,\cdots,m,$ it follows that
\begin{equation*}
|\lambda|>\prod^{m}_{l=1}\max\{(R_{i}(\mathcal C)_{+})^{\frac{1}{m}},(R_{i}(\mathcal C)_{-})^{\frac{1}{m}}\},
\end{equation*}
which contradicts the fact that $\lambda\in\mathcal F(\mathcal C).$
\end{proof}
\begin{dl}\label{dl:2.5}
Let $\mathcal C=(c_{i_{1}\cdots i_{m}})\in\mathbb{ R}^{[m,n]},$ $m,n\geq2.$ Then
$$\sigma(\mathcal C)\subseteq \mathcal B(\mathcal C)=\bigcup_{i,j\in[n], i\neq j}(\mathcal B_{i,j}^{1}(\mathcal C)\bigcup(\mathcal B_{i,j}^{2}(\mathcal C)\bigcap\Omega_{i})),$$
where $\mathcal B_{i,j}^{1}(\mathcal C)=\{\lambda\in\mathbb{R}: |\lambda|\leq\max\{R_{i}(\mathcal C)_{+}-P^{i}_{i}(\mathcal C)_{+}, R_{i}(\mathcal C)_{-}-P^{i}_{i}(\mathcal C)_{-}\}, |\lambda|\leq\max\{R_{j}(\mathcal C)_{+}-P^{j}_{j}(\mathcal C)_{+}, R_{j}(\mathcal C)_{-}-P^{j}_{j}(\mathcal C)_{-}\}\},$
$\mathcal B_{i,j}^{2}(\mathcal C)=\{\lambda\in\mathbb{R}: (|\lambda|-\max\{R_{i}(\mathcal C)_{+}-P^{i}_{i}(\mathcal C)_{+}, R_{i}(\mathcal C)_{-}-P^{i}_{i}(\mathcal C)_{-}\})(|\lambda|-\max\{R_{j}(\mathcal C)_{+}-P^{j}_{j}(\mathcal C)_{+}, R_{j}(\mathcal C)_{-}-P^{j}_{j}(\mathcal C)_{-}\})\leq\max\{P_{i}^{i}(\mathcal C)_{+},P_{i}^{i}(\mathcal C)_{-}\}\times\max\{P_{j}^{j}(\mathcal C)_{+},P_{j}^{j}(\mathcal C)_{-}\}\}.$
\end{dl}
\begin{proof}
Let $\lambda$ be a Pareto Z-eigenvalue of $\mathcal C$ with corresponding eigenvector $x$. Let $|x_{p}|\geq|x_{q}|\geq \max \{x_{i}:i\in[n], i\neq p, i\neq q\}$. Then it is evident that $x_{p}>0.$ From Definition \ref{dy:1.1}, we have
\begin{equation*}
\begin{aligned}
|\lambda|x_{p}^{2}&=\Bigg|\sum_{\substack{i_{2},\cdots,i_{m}\in [n] \\ p\in\{i_{2},\cdots,i_{m}\}}}c_{pi_{2}\cdots i_{m}}x_{p}x_{i_{2}}\cdots x_{i_{m}}+\sum_{\substack{i_{2},\cdots,i_{m}\in [n] \\ p\notin\{i_{2},\cdots,i_{m}\}}}c_{pi_{2}\cdots i_{m}}x_{p}x_{i_{2}}\cdots x_{i_{m}}\Bigg|\\
&\leq\max\left\{\sum_{\substack{i_{2},\cdots,i_{m}\in [n] \\ p\in \{i_{2},\cdots,i_{m}\}}}[c_{pi_{2}\cdots i_{m}}]_{+}x_{p}x_{i_{2}}\cdots x_{i_{m}},\sum_{\substack{i_{2},\cdots,i_{m}\in [n] \\ p\in \{i_{2},\cdots,i_{m}\}}}[c_{pi_{2}\cdots i_{m}}x_{i_{2}}]_{-}x_{p}x_{i_{2}}\cdots x_{i_{m}}\right\}\\
&+\max\left\{\sum_{\substack{i_{2},\cdots,i_{m}\in [n] \\ p\notin \{i_{2},\cdots,i_{m}\}}}[c_{pi_{2}\cdots i_{m}}]_{+}x_{p}x_{i_{2}}\cdots x_{i_{m}},\sum_{\substack{i_{2},\cdots,i_{m}\in [n] \\ p\notin \{i_{2},\cdots,i_{m}\}}}[c_{pi_{2}\cdots i_{m}}]_{-}x_{p}x_{i_{2}}\cdots x_{i_{m}}\right\}\\
&\leq\max\left\{\sum_{\substack{i_{2},\cdots,i_{m}\in [n] \\ p\in \{i_{2},\cdots,i_{m}\}}}[c_{pi_{2}\cdots i_{m}}]_{+},\sum_{\substack{i_{2}\cdots i_{m}\in [n] \\ p\in \{i_{2},\cdots,i_{m}\}}}[c_{pi_{2}\cdots i_{m}}]_{-}\right\}x_{p}^{2}\\
&+\max\left\{\sum_{\substack{i_{2},\cdots,i_{m}\in [n] \\ p\notin \{i_{2},\cdots,i_{m}\}}}[c_{pi_{2}\cdots i_{m}}]_{+},\sum_{\substack{i_{2},\cdots,i_{m}\in [n] \\ p\notin \{i_{2},\cdots,i_{m}\}}}[c_{pi_{2}\cdots i_{m}}]_{-}\right\}x_{p}x_{q}.
\end{aligned}
\end{equation*}
Furthermore, we obtain
\begin{equation*}
\begin{aligned}
&\Big(|\lambda|-\max\Big\{\sum_{\substack{i_{2},\cdots,i_{m}\in [n] \\ p\in \{i_{2},\cdots,i_{m}\}}}[c_{pi_{2}\cdots i_{m}}]_{+},\sum_{\substack{i_{2},\cdots,i_{m}\in [n] \\ p\in \{i_{2},\cdots,i_{m}\}}}[c_{pi_{2}\cdots i_{m}}]_{-}\Big\}\Big)x_{p}^{2}\\
&\leq\max\Big\{\sum_{\substack{i_{2},\cdots,i_{m}\in [n] \\ p\notin \{i_{2},\cdots,i_{m}\}}}[c_{pi_{2}\cdots i_{m}}]_{+},\sum_{\substack{i_{2},\cdots,i_{m}\in [n] \\ p\notin \{i_{2},\cdots,i_{m}\}}}[c_{pi_{2}\cdots i_{m}}]_{-}\Big\}x_{p}x_{q},
\end{aligned}
\end{equation*}
that is
\begin{equation}\label{2.6}
\left(|\lambda|-\max\{R_{p}(\mathcal C)_{+}-P_{p}^{p}(\mathcal C)_{+},R_{p}(\mathcal C)_{-}-P_{p}^{p}(\mathcal C)_{-}\}\right)\leq\max\{P_{p}^{p}(\mathcal C)_{+},P_{p}^{p}(\mathcal C)_{-}\}\frac{x_{q}}{x_{p}}.
\end{equation}
It is easy to get $\lambda\in\Omega_{p}(\mathcal C)$ by virtue of the Inequality \eqref{2.6}.

{\bf Case I.} If $x_{q}=0,$ then $|\lambda|-\max\{R_{p}(\mathcal C)_{+}-P_{p}^{p}(\mathcal C)_{+},R_{p}(\mathcal C)_{-}-P_{p}^{p}(\mathcal C)_{-}\}\leq0.$ If $(|\lambda|-\max\{R_{q}(\mathcal C)_{+}-P_{q}^{q}(\mathcal C)_{+},R_{q}(\mathcal C)_{-}-P_{q}^{q}(\mathcal C)_{-}\})>0,$ then we have
\begin{equation*}
\begin{aligned}
 (|\lambda|-\max\{R_{p}(\mathcal C)_{+}-P_{p}^{p}(\mathcal C)_{+},R_{p}(\mathcal C)_{-}-P_{p}^{p}(\mathcal C)_{-}\})(|\lambda|-\max\{R_{q}(\mathcal C)_{+}-P_{q}^{q}(\mathcal C)_{+},R_{q}(\mathcal C)_{-}-P_{q}^{q}(\mathcal C)_{-}\})
\end{aligned}
\end{equation*}
$$\leq0\leq\max\{P_{p}^{p}(\mathcal C)_{+},P_{p}^{p}(\mathcal C)_{-}\}\times\max\{P_{q}^{q}(\mathcal C)_{+},P_{q}^{q}(\mathcal C)_{-}\},$$
which implies $\lambda\in\mathcal B_{p,q}^{2}(\mathcal C)\bigcap\Omega_{p}(\mathcal C)\subseteq\mathcal B(\mathcal C).$ And if $(|\lambda|-\max\{R_{q}(\mathcal C)_{+}-P_{q}^{q}(\mathcal C)_{+},R_{q}(\mathcal C)_{-}-P_{q}^{q}(\mathcal C)_{-}\})\leq0,$ we obtain $\lambda\in\mathcal B_{p,q}^{1}(\mathcal C)\subseteq\mathcal B(\mathcal C).$

{\bf Case II.} If $x_{q}>0,$ it holds that
\begin{equation*}
\begin{aligned}
|\lambda| x_{q}^{2} &\leq \max\left\{ 
    \sum_{\substack{i_2,\dots,i_m \in [n] \\ q \in \{i_2,\dots,i_m\}}} 
    [c_{q i_2 \cdots i_m}]_{+},\ 
    \sum_{\substack{i_2,\dots,i_m \in [n] \\ q \in \{i_2,\dots,i_m\}}} 
    [c_{q i_2 \cdots i_m}]_{-}
\right\} x_{p} x_{q}\\
&+\max\left\{\sum_{\substack{i_2,\dots,i_m \in [n] \\ q \notin \{i_2,\dots,i_m\}}}[c_{q i_2 \cdots i_m}]_{+},\sum_{\substack{i_2,\dots,i_m \in [n] \\ q\in\{i_2,\dots,i_m\}}}[c_{q i_2 \cdots i_m}]_{-}\right\}x_p^{2}\\
&=\max\{R_{q}(\mathcal C)_{+}-P_{q}^{q}(\mathcal C)_{+},R_{q}(\mathcal C)_{-}-P_{q}^{q}(\mathcal C)_{-}\}x_{p}x_{q}+\max\{P_{q}^{q}(\mathcal C)_{+},P_{q}^{q}(\mathcal C)_{-}\}x_{p}^{2}.
\end{aligned}
\end{equation*}
Furthermore, we obtain
\begin{equation}\label{2.7}
\Big(|\lambda|-\max\{R_{q}(\mathcal C)_{+}-P_{q}^{q}(\mathcal C)_{+},R_{q}(\mathcal C)_{-}-P_{q}^{q}(\mathcal C)_{-}\}\Big)\leq\max\{P_{q}^{q}(\mathcal C)_{+},P_{q}^{q}(\mathcal C)_{-}\}\frac{x_{p}}{x_{q}}.
\end{equation}
If $|\lambda|-\max\{R_{p}(\mathcal C)_{+}-P_{p}^{p}(\mathcal C)_{+},R_{p}(\mathcal C)_{-}-P_{p}^{p}(\mathcal C)_{-}\}>0$ or $|\lambda|-\max\{R_{q}(\mathcal C)_{+}-P_{q}^{q}(\mathcal C)_{+},R_{q}(\mathcal C)_{-}-P_{q}^{q}(\mathcal C)_{-}\}>0.$
We have the following inequality by multiplying \eqref{2.6} and \eqref{2.7} noting that $x_{p}x_{q}>0,$ we get
\begin{equation*}
\begin{aligned}
&\big(|\lambda|-\max\{R_{p}(\mathcal C)_{+}-P_{p}^{p}(\mathcal C)_{+},R_{p}(\mathcal C)_{-}-P_{p}^{p}(\mathcal C)_{-}\}\big)\\
&\big(|\lambda|-\max\{R_{q}(\mathcal C)_{+}-P_{q}^{q}(\mathcal C)_{+},R_{q}(\mathcal C)_{-}-P_{q}^{q}(\mathcal C)_{-}\}\big)\\
&\leq\max\{P_{p}^{p}(\mathcal C)_{+},P_{p}^{p}(\mathcal C)_{-}\}\times\max\{P_{q}^{q}(\mathcal C)_{+},P_{q}^{q}(\mathcal C)_{-}\}.
\end{aligned}
\end{equation*}
This shows $\lambda\in\mathcal B_{p,q}^{2}(\mathcal C)\bigcap\Omega(\mathcal C)\subseteq\mathcal B(\mathcal C).$
If $|\lambda|\leq\max\{R_{p}(\mathcal C)_{+}-P_{p}^{p}(\mathcal C)_{+},R_{p}(\mathcal C)_{-}-P_{p}^{p}(\mathcal C)_{-}\}$ and  $|\lambda|\leq\max\{R_{q}(\mathcal C)_{+}-P_{q}^{q}(\mathcal C)_{+},R_{q}(\mathcal C)_{-}-P_{q}^{q}(\mathcal C)_{-}\},$ then $\lambda\in\mathcal B_{p,q}^{1}(\mathcal C)\bigcap\Omega(\mathcal C)\subseteq\mathcal B(\mathcal C).$
Therefore, the conclusion $\lambda\in\bigcup_{i,j\in[n], i\neq j}\left(\mathcal B_{i,j}^{1}(\mathcal C)\bigcup(\mathcal B_{i,j}^{2}(\mathcal C)\bigcap\Omega_{i})\right)$ follows directly from the above results, which contradicts the fact that $\lambda\in\mathcal F(\mathcal C).$
\end{proof}
\noindent{\bf Example 1} Consider $\mathcal C=(c_{i_{1}i_{2}i_{3}})\in\mathbb{R}^{[3,3]}$ defined by\\ 
$c_{111} = 1, c_{112} = -1, c_{131} = 1,  c_{133} = 1, c_{211} = -1,  c_{222} = 2, c_{232} = 1,$ $c_{311} = 1,  c_{322} = 3, c_{323} = 1,$\\ $c_{ijk} = 0$, and other entries be all zero.
By calculating, we have\\
Based on Theorem 1 in \cite{Yang1}, we conclude that
\begin{equation*}
  \Psi(\mathcal{C})=\{\lambda \in \mathbb{R}: -\sqrt{2} \le \lambda \le \sqrt{19}\}.
\end{equation*}
In light of Theorem 2 in \cite{Yang1}, we obtain
\begin{equation*}
\Omega(\mathcal{C})=\bigcup_{i=1}^3 \Omega_i(\mathcal{C}) = \{\lambda \in \mathbb{R}: |\lambda| \le 5\}.
\end{equation*}
From Theorem 3 in \cite{Yang1}, it follows that
\begin{equation*}
\Phi(\mathcal{C})=\bigcup_{i \in [n]} \bigcap_{j \in [n], j \neq i} \Phi_{ij}(\mathcal{C}) 
                    = \{\lambda \in \mathbb{R}: |\lambda| \le 1 + \sqrt{10}\}.
\end{equation*} 
Based on Theorem 3.1 in \cite{Wang2},  it holds that
\begin{equation*}
\gamma(\mathcal C)= \{\lambda \in \mathbb{R}: -\frac{\sqrt{2}}{2}\leq\lambda\leq4.0285\}.
\end{equation*}                                    
By Theorem 2.3 in \cite{He2}, we get 
\begin{equation*}
\sum(\mathcal{K})= \{\lambda \in \mathbb{R}: -2\leq\lambda\leq5\}.
\end{equation*}
From Theorem \ref{dl:2.1}, it follows that
\begin{equation*}
\mathcal O(\mathcal C)=\bigcup_{i\in[n]}\bigcap_{j\in[n] \ i\neq j}\mathcal O_{i,j}(\mathcal C)=\{\lambda \in \mathbb{R}: |\lambda| \leq 4\}.
\end{equation*}
\begin{table}[htbp]
\centering
\begin{tabular}{@{}l@{\hspace{1em}}l@{}}
\hline
$\mathcal O_{1,2}: |\lambda| \leq {\small \frac{1+\sqrt{37}}{2}}$ & $\mathcal O_{1,3}: |\lambda| \leq 1+\sqrt{6}$ \\
$\mathcal O_{2,1}: |\lambda| \leq {\small \frac{1+\sqrt{37}}{2}}$ & $\mathcal O_{2,3}: |\lambda| \leq {\small \frac{1+\sqrt{41}}{2}}$ \\
$\mathcal O_{3,1}: |\lambda| \leq   4$ & $\mathcal O_{3,2}: |\lambda| \leq 2+\sqrt{7}$ \\
\hline
\end{tabular}
\end{table}

\noindent Base on Theorem \ref{dl:2.3}, It holds that
\begin{equation*}
\mathcal F(\mathcal C)=\sqrt[3]{75}.
\end{equation*}
It follows from Theorem \ref{dl:2.5} that
\begin{equation*}
 \mathcal B(\mathcal C)=\bigcup_{i,j\in[n], i\neq j}(\mathcal B_{i,j}^{1}(\mathcal C)\bigcup(\mathcal B_{i,j}^{2}(\mathcal C)\bigcap\Omega_{i}(\mathcal C)))=\{\lambda \in \mathbb{R}: |\lambda| \leq 2+\sqrt{15}\}.
\end{equation*}
\begin{table}[htbp]
\centering
\begin{tabular}{@{}l@{\hspace{1em}}l@{}}
\hline
$\mathcal B_{1,2}: |\lambda| \leq 3$ & $\mathcal B_{1,3}: |\lambda| \leq 3$ \\
$\mathcal B_{2,1}: |\lambda| \leq 3$ & $\mathcal B_{2,3}: |\lambda| \leq 3$ \\
$\mathcal B_{3,1}: |\lambda| \leq  {\small \frac{3+\sqrt{17}}{2}}$ & $\mathcal B_{3,2}: |\lambda| \leq 2+\sqrt{15}$ \\
\hline
\end{tabular}
\end{table}

Example 1 demonstrates that the newly proposed Pareto Z-eigenvalue localization sets in this paper are shaper than some existing ones\cite{Yang1, Wang2, He2}.
\section{Applications}
\subsection{A sufficient criteria for judging copositivity (strict copositivity)}
Based on the Pareto Z-eigenvalue localization intervals in Theorem \ref{dl:2.1}, we give a sufficient condition for a real symmetric tensor to be copositive (strictly copositive), which are presented in the following theorem.
\begin{dl}\label{dl:3.1}
Let $\mathcal C=(c_{i_{1}\cdots i_{m}})\in\mathbb{ R}^{[m,n]}, m, n\geq2$ and $c_{i\cdots i}>0.$ If the following condition holds\\
\begin{equation}\label{3.1}
\begin{aligned}
&c_{ii\cdots i}+\frac{1}{{(m-1)}^\frac{m-1}{2}}\sum_{\substack{i_{2},\cdots,i_{m}\in [n] \\\delta_{ii_{2}\cdots i_{m}=0}\\ j\in\{i_{2},\cdots,i_{m}\}}}[c_{ii_{2}\cdots i_{m}}]_{+}\\
&+\frac{1}{{(m-1)}^\frac{m-1}{2}}\sum_{\substack{i_{2},\cdots,i_{m}\in [n] \\\delta_{ii_{2}\cdots i_{m}=0}\\ j\notin\{i_{2},\cdots,i_{m}\}}}[c_{ii_{2}\cdots i_{m}}]_{+}-R_{i}(\mathcal C)_{-}\geq(>)0,
\end{aligned}
\end{equation}
for all $i,j\in[n], i\neq j,$ then $\mathcal C$ is copositive (strictly copositive).
\end{dl}
\begin{proof}
Suppose that $(\lambda,x)$ is a Pareto Z-eigenpair of $\mathcal C$. Let $0<x_{p}= \max_{i\in[n]}x_{i}$ and according to Definition \ref{dy:1.1}, it holds that
\begin{equation*}
\begin{aligned}
\lambda x_{p}^{2}&=c_{pp\cdots p}x_{p}^{m}+\sum_{\substack{i_{2},\cdots,i_{m}\in [n] \\\delta_{pi_{2}\cdots i_{m}=0} \\ q\in\{i_{2},\cdots,i_{m}\}}}[c_{pi_{2}\cdots i_{m}}]_{+}x_{p}x_{i_{2}}\cdots x_{i_{m}}-\sum_{\substack{i_{2},\cdots,i_{m}\in [n] \\\delta_{pi_{2}\cdots i_{m}=0}\\ q\in\{i_{2},\cdots,i_{m}\}}}[c_{pi_{2}\cdots i_{m}}]_{-}x_{p}x_{i_{2}}\cdots x_{i_{m}}\\
&+\sum_{\substack{i_{2},\cdots,i_{m}\in [n] \\\delta_{pi_{2}\cdots i_{m}=0}\\ q\notin\{i_{2},\cdots,i_{m}\}}}[c_{pi_{2}\cdots i_{m}}]_{+}x_{p}x_{i_{2}}\cdots x_{i_{m}}-\sum_{\substack{i_{2},\cdots,i_{m}\in [n] \\\delta_{pi_{2}\cdots i_{m}=0} \\ q\notin\{i_{2},\cdots,i_{m}\}}}[c_{pi_{2}\cdots i_{m}}]_{-}x_{p}x_{i_{2}}\cdots x_{i_{m}}.
\end{aligned}
\end{equation*}
From Lemma \ref{2.2}, we obtain
\begin{equation*}
\begin{aligned}
\lambda x_{p}^{2} &\geq c_{pp\cdots p}x_{p}^{2}+\sum_{\substack{i_{2},\cdots,i_{m}\in [n] \\\delta_{pi_{2}\cdots i_{m}=0}\\ q\in\{i_{2},\cdots,i_{m}\}}}[c_{pi_{2}\cdots i_{m}}]_{+}x_{p}x_{i_{2}}\cdots x_{i_{m}}-\sum_{\substack{i_{2},\cdots,i_{m}\in [n] \\\delta_{pi_{2}\cdots i_{m}=0}\\ q\in\{i_{2},\cdots,i_{m}\}}}[c_{pi_{2}\cdots i_{m}}]_{-}x_{p}^{2}\\
&+\sum_{\substack{i_{2},\cdots,i_{m}\in [n] \\\delta_{pi_{2}\cdots i_{m}=0}\\ q\notin\{i_{2},\cdots,i_{m}\}}}[c_{pi_{2}\cdots i_{m}}]_{+}x_{p}x_{i_{2}}\cdots x_{i_{m}}-\sum_{\substack{i_{2},\cdots,i_{m}\in [n]\\\delta_{pi_{2}\cdots i_{m}=0} \\ q\notin\{i_{2},\cdots,i_{m}\}}}[c_{pi_{2}\cdots i_{m}}]_{-}x_{p}^{2}.
\end{aligned}
\end{equation*}
Furthermore, we have
\begin{equation*}
\begin{aligned}
\Big(\lambda-c_{pp\cdots p}+\sum_{\substack{i_{2},\cdots,i_{m}\in [n] \\\delta_{pi_{2}\cdots i_{m}=0}\\ q\in\{i_{2},\cdots,i_{m}\}}}[c_{pi_{2}\cdots i_{m}}]_{-}
+\sum_{\substack{i_{2},\cdots,i_{m}\in [n] \\\delta_{pi_{2}\cdots i_{m}=0}\\ q\notin\{i_{2},\cdots,i_{m}\}}}[c_{pi_{2}\cdots i_{m}}]_{-}\Big)x_{p}\\
\geq\frac{1}{{(m-1)}^\frac{m-1}{2}}\sum_{\substack{i_{2},\cdots,i_{m}\in [n] \\\delta_{pi_{2}\cdots i_{m}=0}\\ q\in\{i_{2},\cdots,i_{m}\}}}[c_{pi_{2}\cdots i_{m}}]_{+}
+\frac{1}{{(m-1)}^\frac{m-1}{2}}\sum_{\substack{i_{2},\cdots,i_{m}\in [n] \\\delta_{pi_{2}\cdots i_{m}=0}\\ q\notin\{i_{2},\cdots,i_{m}\}}}[c_{pi_{2}\cdots i_{m}}]_{+}.
\end{aligned}
\end{equation*}
Since 
\begin{equation*}
\begin{aligned}
\lambda-c_{pp\cdots p}+\sum_{\substack{i_{2},\cdots,i_{m}\in [n] \\\delta_{pi_{2}\cdots i_{m}=0}\\ q\in\{i_{2},\cdots,i_{m}\}}}[c_{pi_{2}\cdots i_{m}}]_{-}
+\sum_{\substack{i_{2},\cdots,i_{m}\in [n] \\\delta_{pi_{2}\cdots i_{m}=0}\\ q\notin\{i_{2},\cdots,i_{m}\}}}[c_{pi_{2}\cdots i_{m}}]_{-}\\
\geq\Big(\lambda-a_{pp\cdots p}+\sum_{\substack{i_{2},\cdots,i_{m}\in [n] \\\delta_{pi_{2}\cdots i_{m}=0}\\ q\in\{i_{2},\cdots,i_{m}\}}}[c_{pi_{2}\cdots i_{m}}]_{-}
+\sum_{\substack{i_{2},\cdots,i_{m}\in [n] \\\delta_{pi_{2}\cdots i_{m}=0}\\ q\notin\{i_{2},\cdots,i_{m}\}}}[c_{pi_{2}\cdots i_{m}}]_{-}\Big)x_{p},
\end{aligned}
\end{equation*}
we get 
\begin{equation*}
\begin{aligned}
\lambda &\geq c_{pp\cdots p}+\frac{1}{{(m-1)}^\frac{m-1}{2}}\sum_{\substack{i_{2},\cdots,i_{m}\in [n] \\\delta_{pi_{2}\cdots i_{m}=0}\\ q\in\{i_{2},\cdots,i_{m}\}}}[c_{pi_{2}\cdots i_{m}}]_{+}-\sum_{\substack{i_{2},\cdots,i_{m}\in [n] \\\delta_{pi_{2}\cdots i_{m}=0}\\ q\in\{i_{2},\cdots,i_{m}\}}}[c_{pi_{2}\cdots i_{m}}]_{-}\\
&+\frac{1}{{(m-1)}^\frac{m-1}{2}}\sum_{\substack{i_{2},\cdots,i_{m}\in [n] \\\delta_{pi_{2}\cdots i_{m}=0}\\ q\notin\{i_{2},\cdots,i_{m}\}}}[c_{pi_{2}\cdots i_{m}}]_{+}-\sum_{\substack{i_{2},\cdots,i_{m}\in [n] \\\delta_{pi_{2}\cdots i_{m}=0}\\ q\notin\{i_{2},\cdots,i_{m}\}}}[c_{pi_{2}\cdots i_{m}}]_{-},
\end{aligned}
\end{equation*}
that is
\begin{equation}\label{3.3}
\begin{aligned}
\lambda &\geq c_{pp\cdots p}+\frac{1}{{(m-1)}^\frac{m-1}{2}}\sum_{\substack{i_{2},\cdots,i_{m}\in [n] \\\delta_{pi_{2}\cdots i_{m}=0}\\ q\in\{i_{2},\cdots,i_{m}\}}}[c_{pi_{2}\cdots i_{m}}]_{+}\\
&+\frac{1}{{(m-1)}^\frac{m-1}{2}}\sum_{\substack{i_{2},\cdots,i_{m}\in [n] \\\delta_{pi_{2}\cdots i_{m}=0}\\ q\notin\{i_{2},\cdots,i_{m}\}}}[c_{pi_{2}\cdots i_{m}}]_{+}-R_{p}(\mathcal C)_{-}.
\end{aligned}
\end{equation}

Multiplying (\ref{3.1}) and (\ref{3.3}) implies $\lambda\geq(>)0$, hence $\mathcal C$ is (copositive copositive).
\end{proof}
The following example demonstrates that Theorem \ref{3.1} can be effectively applied to verify the strict copositivity of tensors.\\
{\bf Example 2} Consider $\mathcal C=(c_{i_{1}i_{2}i_{3}})\in\mathbb{R}^{[3,2]}$ defined by\\
$c_{111}=4, c_{112}=-1, c_{121}=-1, c_{122}=0, c_{222}=2, c_{211}=-1, c_{212}=0, c_{221}=0$.\\
From Theorem \ref{3.1}, it holds that\\
\begin{equation*}\quad
i=1, \quad j=2, \quad 2>0,
\end{equation*}
\begin{equation*}\quad
i=2, \quad j=1, \quad 1>0,
\end{equation*}
which means that $\mathcal C$ is strictly copositive.
\subsection{A sufficient condition for the solution of the TCP and TGEiCP }
The tensor complementarity problem, denoted by TCP$(p,\mathcal A)$, is to find $x\in\mathbb{R}^{n}$ such that
\begin{equation}\label{3.3}
x\geq0, \qquad p+\mathcal Cx^{m-1}\geq0, \qquad \mbox{and} \qquad x^{\top}( p+\mathcal Cx^{m-1})=0,
\end{equation}
which is also a generalization of the linear complementarity problem. The tensor complementarity problem TCP$(p,\mathcal C)$  has found numerous applications in DNA micro-arrays, communication systems, and n-person non-cooperative games \cite{Huang1, Luo1}, attracting considerable research attention in recent years.

Copositive tensors constitute a fundamental class in the study of tensor complementarity problems. Che et al.\cite{Che1} established that the tensor complementarity problem with a copositive tensor (strictly copositive tensor) necessarily possesses a nonempty and compact solution set. Song et al.\cite{Song2} proved the equivalence between semi-positivity (strict semi-positivity) and copositivity (strict copositivity) for real symmetric tensors. Building upon this equivalence, Song et al.\cite{Song2, Song3} further developed a series of results for tensor complementarity problems with semi-positive tensors (strictly semi-positive tensors).

\begin{dl}\label{dl:3.2}(\cite{Song2, Song3, Che2}
Let $\mathcal C\in\mathbb{R}^{[m,n]}$ be a given symmetric tensor. Then the following results hold.\\ 
(1) If $\mathcal C$ is strictly copositive, then the TCP$(p,\mathcal C)$ has a nonempty, compact solution set;\\
(2) If $\mathcal C$ is strictly copositive if and only if the TCP$(p,\mathcal C)$ has a unique solution for every $p\geq0$.
\end{dl}
According to Theorem \ref{dl:3.1} and \ref{dl:3.2}, we establish a sufficient condition for the existence and uniqueness of a global solution to the TCP.
\begin{tl}\label{tl:3.1}
Let $\mathcal C=(c_{i_{1}i_{2}\cdots i_{m}})\in\mathbb{R}^{[m,n]}$, if the condition in Theorem \ref{dl:3.1} is satisfied, then the TCP$(p, \mathcal C)$ has a nonempty,  compact solution set.
\end{tl}
\begin{tl}\label{tl:3.2}
Let $\mathcal C=(c_{i_{1}i_{2}\cdots i_{m}})\in\mathbb{R}^{[m,n]}$, if the condition in Theorem \ref{dl:3.1} is satisfied, then the TCP$(p, \mathcal C)$ has a unique solution for every $p\geq0$.
\end{tl}
The tensor generalized eigenvalue complementarity problem (TGEiCP), is to find $x\in\mathbb{R}^{n}$ such that
\begin{equation*}
\begin{cases}
x \geq 0, \\
\lambda \mathcal{C}_2 x^{m-1} - \mathcal{C}_1 x^{m-1} \geq 0, \\
x^\top \bigl( \mathcal{C}_2 x^{m-1} - \mathcal{C}_1 x^{m-1} \bigr) = 0.
\end{cases}
\end{equation*}
Suppose that $\lambda^{\max}_{\mathcal C_{1},\mathcal C_{2}}=\max\{\lambda|x\in\mathbb{R}_{+}^{n}, (\lambda,x)\ \mbox{is a solution of TGEiP}\},$
\begin{equation}\label{3.4}
\rho(\mathcal C_{1}, \mathcal C_{2})=\max_{x}\{\lambda(x),\mathcal C_{2}x^{m}\neq0\},
\end{equation}
where $\lambda(x)=\frac{\mathcal C_{1}x^{m}}{\mathcal C_{2}x^{m}}$. 

Ling et al.\cite{Ling1} discussed the solution existence of the TGEiCP as follows.
\begin{dl}\label{dl:3.3}(\cite{Ling1})
Let $\mathcal C_{1}, \mathcal C_{2}\in \mathbb{R}^{[m,n]}$. The following results hold.\\
(1) If $\mathcal C_{2}$ is strictly copositive, then TGEiCP has at least one solution;\\
(2) If $\mathcal C_{1}$ and $\mathcal C_{2}$ are symmetric and $\mathcal C_{2}$ is strictly copositive, and if $\bar{x}$ is a stationary point of (\ref{3.4}), then $(\lambda(\bar{x}), \bar{x})$ is a solution of TGEiCP.
\end{dl}
Based on Theorem \ref{dl:3.1} and \ref{dl:3.3}, a sufficient condition for the existence of a solution to the TGEiCP can be obtained directly.
\begin{tl}\label{tl:3.3}
Let $\mathcal C_{1}, \mathcal C_{2} \in \mathbb{R}^{[m,n]}$. The following results hold:\\
(1) If $\mathcal C_{2}$ satisfies the condition in Theorem \ref{dl:3.1}, then the TGEiCP has at least one solution;\\
(2) If $\mathcal C_{1}$ and $\mathcal C_{2}$ are symmetric, and $\mathcal C_{2}$ satisfies the condition of Theorem \ref{dl:3.1}, and if $\bar{x}$ is a stationary point of (\ref{3.4}), then $(\lambda(\bar{x}), \bar{x})$ is a solution of TGEiCP.
\end{tl}
\subsection{Checking vacuum stability of $\mathbb{Z}_3$ Scalar Dark Matter }
The vacuum stability of scalar potentials involving multiple fields was systematically investigated by Kannike in \cite{K1}, where explicit necessary and sufficient conditions were derived for ensuring stability in intricate multi-dimensional field spaces.
We consider the most general scalar quartic potential of the standard model Higgs $H_{1}$, an inert doublet $H_{2}$ and a complex singlet $S$ which is symmetric under a $\mathbb{Z}_3$ discrete group\cite{K1}:
\begin{equation}\label{3.5}
\begin{aligned}
V(H_1, H_2, S) &= \lambda_{1} \|H_{1}\|^{4} + \lambda_{2} \|H_{2}\|^{4} + \lambda_{3} \|H_{1}\|^{2} \|H_{2}\|^{2} \\
&\quad + \lambda_{4} (H_{1}^\dagger H_{2})(H_{2}^\dagger H_{1}) + \lambda_{S} \|S\|^{4} + \lambda_{S1} \|S\|^{2} \|H_{1}\|^{2} \\
&\quad + \lambda_{S2} \|S\|^{2} \|H_{2}\|^{2} + \frac{1}{2} (\lambda_{S12} S^{2} H_{1}^\dagger H_{2} + \lambda_{S12}^* S^{\dagger} H_{2}^\dagger H_{1}) \\
=& \lambda_1 h_1^4+\lambda_2 h_2^4+\lambda_3 h_1^2 h_2^2+\lambda_4 \rho^2 h_1^2h_2^2+\lambda_Ss^4+\lambda_{S1}s^2h_1^2\\&+\lambda_{S2}s^2h_2^2+|\lambda_{S12}|\rho s^2h_1h_2\cos (\theta +2\theta_S+\theta_\lambda),
\end{aligned}
\end{equation}
where \begin{center}
$h_{1}=\|H_{1}\|$, $h_{2}=\|H_{2}\|$, $H_1^{\dagger}H_2=h_{1}h_{2}\rho \textbf{e}^{\textbf{i}\theta}$, $S=s\textbf{e}^{\textbf{i}\theta_{S}}$, $\lambda_{S12}=|\lambda_{S12}|\textbf{e}^{\textbf{i}\theta_{\lambda}}$. 
\end{center} This definites  a 4th order 3-dimensional real symmetric tensor $\mathcal{T}_\rho^{\theta'}=(t_{ijkl})$  with two parameters $\rho$ and $\theta'$  (here $\theta'=\theta +2\theta_S+\theta_\lambda$). Obviously, it follows that $V(H_{1},H_{2},S)\geq0$ for all $H_{1}, H_{2}, S$ is equivalent to  $V(h_{1},h_{2},s)\geq0$ for all $\rho\in[0,1], h_{1}, h_{2}, s$  (take the minimizer $\cos \theta'=-1$),
\begin{equation}\label{eq:36} \aligned	
V(h_{1},h_{2},s)&=\lambda_{1}h_{1}^{4}+\lambda_{2}h_{2}^{4}+\lambda_{3}h_{1}^{2}h_{2}^{2}+\lambda_{4}\rho^{2}h_{1}^{2}h_{2}^{2}\\
&\quad+\lambda_{S}s^{4}+\lambda_{S1}s^{2}h_1^2+\lambda_{S2}s^{2}h_{2}^{2}-|\lambda_{S12}|\rho s^{2}h_{1}h_{2}.
\endaligned\end{equation} Consequently, the coupling tensor $\mathcal{V} = (v_{p_{1}p_{2}p_{3}p_{4}})$ of coefficients of \eqref{eq:36} is a real symmetric tensor of order $4$ and dimension $3$, with
\begin{equation*}
\begin{split}
&v_{1111} = \lambda_1, \quad v_{2222} = \lambda_2, \quad v_{3333} = \lambda_S, v_{1122} = \frac{1}{6}(\lambda_3 + \lambda_4 \rho^2),\\
&v_{1133} = \frac{1}{6}\lambda_{S1}, \quad v_{2233} = \frac{1}{6}\lambda_{S2}, v_{1233} = -\frac{1}{12}\lvert\lambda_{S12}\rvert,
\end{split}
\end{equation*}
and others $v_{p_{1}p_{2}p_{3}p_{4}}=0$. Note that the notations below are all same as given in \cite{Chen1}.

When $\rho \neq 0$, Kannike \cite{K1} has determined the conditions under which the potential (\ref{3.5}), which is symmetric with respect to a $\mathbb{Z}_{3}$ transformation, is guaranteed to be bounded from below.

\begin{equation}\label{3.6}
\begin{cases}
\lambda_{S}> 0, V(h_{1}, h_{2}):=V(h_{1}, h_{2}, 0)> 0, \\
0 < h_{1}^{2} < 1,\; 0 < h_{2}^{2} < 1,\; 0 < s^{2} < 1, \text{ and } 0 < \rho^{2} < 1 \Rightarrow V_{\min} = 0,
\end{cases}
\end{equation}
where
\begin{equation*}
\begin{aligned}
\rho&=\frac{|\lambda_{S12}| s^{2}}{2 \lambda_4 h_{1} h_{2}},\\
h_{1}^{2}&=\frac{(2\lambda_{2} - \lambda_{3})(4\lambda_{S} \lambda_{4} - |\lambda_{S12}|^{2}) + 2\lambda_{4}\bigl((\lambda_{3} + \lambda_{S1})\lambda_{S2} - 2\lambda_{2} \lambda_{S1} - \lambda_{S2}^2\bigr)}{2t},\\
h_{2}^{2}&=\frac{(2\lambda_{1} - \lambda_{3})(4\lambda_{S}\lambda_{4} - |\lambda_{S12}|^{2}) + 2\lambda_{4}\bigl((\lambda_{3} + \lambda_{S2})\lambda_{S1} - 2\lambda_1\lambda_{S2} - \lambda_{S1}^2\bigr)}{2t},\\
s^{2}&=\frac{\lambda_{4}\bigl(4\lambda_{1}\lambda_{2} - \lambda_{3}^{2} - 2\lambda_{1}\lambda_{S2} - 2\lambda_{2}\lambda_{S1} + \lambda_{3}(\lambda_{S1} + \lambda_{S2})\bigr)}{t},\\
V_{\min}&=\frac{(4\lambda_{1}\lambda_{2} - \lambda_{3}^{2})(4\lambda_{S}\lambda_{4} - |\lambda_{S12}|^{2}) - 4\lambda_{4}\bigl(\lambda_1\lambda_{S2}^{2} + \lambda_{2}\lambda_{S1}^{2} - \lambda_{3}\lambda_{S1}\lambda_{S2}\bigr)}{4t},
\end{aligned}
\end{equation*}
where
\begin{equation*}
t := (\lambda_1 + \lambda_2 - \lambda_3)(4\lambda_S\lambda_4 - |\lambda_{S12}|^2) 
      + \lambda_4\bigl(4\lambda_1\lambda_2 - \lambda_3^2 - 4(\lambda_1\lambda_{S2} + \lambda_2\lambda_{S1}) 
      + 2\lambda_3(\lambda_{S1} + \lambda_{S2}) - (\lambda_{S1} - \lambda_{S2})^2\bigr).
\end{equation*}

The third formula in (\ref{3.6}) is superseded by  $V_{\rho=0} > 0$ when $\rho = 0,$ and by $V_{\rho=1} > 0$  when $\rho = 1.$ 
In \cite{Chen1}, two algorithms were presented to verify the copositivity of tensors, which is essential for verifying the stability of the vacuum state in a generic scalar potential.

Subsequently, an example drawn from \( \mathbb{Z}_{3} \) scalar dark matter is provided to illustrate the application of our findings. Consider the case that
\begin{equation*}
\begin{cases}
    \lambda_{1}= \lambda_{2} = \lambda_{3} = \lambda_{4} = \lambda_{S} = \lambda_{S1} = \lambda_{S2}=\lambda_{S12}=\pi,\\
    h_{1}^{2} = h_{2}^{2} = 0.25,\quad {s}^{2} = 0.5,\quad \rho = 1,\quad V_{\min} = 1.96,
\end{cases}
\end{equation*}
also see the first case in Table 4 of \cite{Chen1}. According to Theorem \ref{3.1}, it can be categorized as follows:\\
If $i=1, j=2,$ we obtain
\begin{equation*}\quad
\pi(\frac{3}{4}+\frac{1}{2\sqrt{3}})>0.
\end{equation*} 
If $i=1, j=3,$ we obtain
\begin{equation*}\quad
\pi(\frac{3}{4}+\frac{1}{2\sqrt{3}})>0.
\end{equation*}
If $i=2, j=1,$ we obtain
\begin{equation*}\quad
\pi(\frac{3}{4}+\frac{1}{2\sqrt{3}})>0.
\end{equation*}
If $i=2, j=3,$ we obtain
\begin{equation*}\quad
\pi(\frac{3}{4}+\frac{1}{2\sqrt{3}})>0.
\end{equation*}
If $i=3, j=1,$ we obtain
\begin{equation*}\quad
\pi(\frac{1}{2}+\frac{1}{3\sqrt{3}})>0.
\end{equation*}
If $i=3, j=2,$ we obtain
\begin{equation*}\quad
\pi(\frac{1}{2}+\frac{1}{3\sqrt{3}})>0.
\end{equation*}

Consequently, $\mathcal V$ constitutes a strictly copositive tensor. This characterization not only renders our preceding results pertinent to such physical systems but also underscores the utility of this framework as a diagnostic tool for copositivity, thereby facilitating the assessment of vacuum stability within the context of a general scalar potential.
\section{Conclusion}
\indent
This paper proposes some Pareto Z-eigenvalue localization sets for the tensor eigenvalue complementarity problem and gives theoretical comparisons with existing results. The theoretical results show that the new Pareto Z-eigenvalue localization sets are sharper than those in Theorem \ref{dl:1.1}. As applications, we present a sufficient condition ensuring the existence of solutions to TCP and TGEiCP. This condition provides an approach to check vacuum stability for a general scalar potential. Finally, numerical experiments are performed to show the effectiveness of the main results.
\section*{Statements and Declarations}
The authors of this work don't have financial or non-financial interests.

\end{document}